\documentclass[reqno]{amsart}
\usepackage[foot]{amsaddr}

\usepackage[utf8]{inputenc}

\usepackage[hidelinks]{hyperref}
\usepackage{url}

\usepackage{amssymb}
\usepackage{amscd}
\usepackage{amsthm}
\usepackage{setspace}
\usepackage{enumerate}
\usepackage{graphicx}
\usepackage{mathtools}
\usepackage{tcolorbox} 
\usepackage{siunitx}
\usepackage{tikz}
\usepackage{pgfplots}
\usepackage{bm,blkarray}
\usepackage{subfig}

\usepackage{mathtools}
\usepackage[tableposition=top]{caption}
\usepackage{booktabs,dcolumn}

\usepackage[obeyFinal,textsize=footnotesize]{todonotes}

\newcommand{\quash}[1]{}

\renewcommand\mod[1]{\ (\mathop{\rm mod}#1)}

\renewcommand{\mod}{\bmod}
 
\title[Maximum Spread of Non-Negative Matrices]{On the Maximum Spread of Non-Negative Matrices}
\author{Susie Lu}
\address{Department of Mathematics, Massachusetts Institute of Technology, \newline \indent Cambridge, MA, 02139 USA.}
\email{susielu@mit.edu}
\author{John Urschel}
\email{urschel@mit.edu}
\newtheorem{theorem}{Theorem}

\newtheorem{lemma}[theorem]{Lemma}

\newtheorem{corollary}[theorem]{Corollary}

\begin{document}

\begin{abstract}
Given a directed graph $G$, the spread of $G$ is the largest distance between any two eigenvalues of its adjacency matrix. In 2022, Breen, Riasanovsky, Tait, and Urschel  asked what $n$-vertex directed graph maximizes spread, and whether this graph is undirected. We prove the more general result that the spread of any $n \times n$ non-negative matrix $A$ with $\|A\|_{\max} \le 1$ is at most $2n/\sqrt{3}$, which is tight up to an additive factor and exact when $n$ is a multiple of three. Furthermore, our results show that the matrix with maximum spread is always symmetric.
\end{abstract}

\maketitle
The spread of an $n \times n$ complex matrix $M$ is the diameter of its spectrum $\Lambda(M)$: that is,
$$\mathrm{diam}(\Lambda(M)) = \max_{\lambda, \lambda' \in \Lambda(M)} |\lambda - \lambda'|.$$
General upper and lower bounds on $s(M)$ have been established by Mirsky \cite[Theorem 2]{Mirsky65} and by Johnson, Kumar, and Wolkowicz \cite[Theorem 2.1]{Johnson85}. 

Consider a (possibly directed) graph $G = (V, E)$ of order $n$. The adjacency matrix $A$ of $G$ is the $n \times n$ matrix whose rows and columns are indexed by the vertices of $G$, with entries satisfying $A_{i,j} = 1$ if $(i,j) \in E$ and $A_{i,j} = 0$ otherwise. The spread of the adjacency matrix of a graph $G$ is
$$ \mathrm{diam}(\Lambda(A)) = \max_{\lambda, \lambda' \in \Lambda(A)} |\lambda - \lambda'|, $$
which is called \emph{the spread of the graph $G$}. In the special case when the graph $G$ is undirected, the adjacency matrix is symmetric, so its eigenvalues are real and can be ordered $\lambda_1(A) \ge \lambda_2(A) \ge \dots \ge \lambda_n(A)$. Then, $\mathrm{diam}(\Lambda(A))  = \lambda_1(A) - \lambda_n(A)$. In 2001, Gregory, Hershkowitz, and Kirkland initiated a detailed analysis of the spread of a graph \cite{gregory2001spread}, leading to a large literature on the subject under various conditions and matrix representations. See, for instance, \cite{andrade2019new,andrade2019bounds,bao2009laplacian,chen2009laplacian,fan2012edge,fan2008laplacian,li2007spread,lin2020aalpha,lin2020bounds,linz2023maximum,liu2009spread,liu2010signless,oliveira2010bounds,xu2011laplacian,you2017distance} and the references therein.

In 2022, Breen, Riasanovsky, Tait, and Urschel characterized the graphs with maximum spread among all $n$-vertex undirected graphs, both allowing and not allowing loops \cite{Breen22}, which solved a main conjecture of Gregory, Hershkowitz, and Kirkland for $n$ sufficiently large. To state their results, we let $\|A\|_{\max}= \max_{i,j} |A_{ij}|$ denote the largest magnitude entry in a matrix $A$, and $A \otimes B$ denote the Kronecker product of two matrices $A$ and $B$.

\begin{theorem}[{\cite[Theorem 1.1 \& 1.3, Corollary 1.4]{Breen22}}]\label{thm:spread}
Let $A$ be an $n \times n$ symmetric non-negative matrix. Then
\[\lambda_{1}(A) - \lambda_{n}(A) \le \frac{2 \,n}{\sqrt{3}}  \, \|A\|_{\max},\]
with equality if and only if $n = 0 \mod 3$ and $A = P_{\pi} \left(\begin{psmallmatrix} 1 & 1 & 1 \\ 1 & 1 & 1 \\ 1 & 1 & 0  \end{psmallmatrix} \otimes J_{\frac{n}{3}}\right)P_{\pi}$, where $J_k$ is the $k \times k$ all-ones matrix and $P_{\pi}$ is a permutation matrix. Furthermore:
\begin{enumerate}
\item Any matrix $A$ that maximizes $\big(\lambda_{1}(A) - \lambda_{n}(A)\big)$ over all $n \times n$ symmetric non-negative matrices with $\|A\|_{\max} = 1$ is a $(0,1)$ matrix and has spread at least $2n/\sqrt{3} - 1/(2 \sqrt{3} n)$.
\item The matrix $A$ that maximizes $\big(\lambda_{1}(A) - \lambda_{n}(A)\big)$ over all $n \times n$ zero-diagonal symmetric non-negative matrices with $\|A\|_{\max} =1$ is exactly the adjacency matrix of the join of a clique on $\lfloor 2n/3 \rfloor$ vertices and an independent set on $\lceil n/3 \rceil$ vertices for all $n$ sufficiently large. This matrix always has spread at least $(2n -1)/\sqrt{3}$ for $n>1$.
\end{enumerate}
\end{theorem}

In the same paper, the authors detailed a number of open problems, beginning with the following: 
\begin{quote}
    ``What digraph of order $n$ maximizes the spread of its adjacency matrix ... Is this
more general problem also maximized by the same set of graphs as in the undirected
case?" \cite[Sec. 8]{Breen22}
\end{quote}

While many refinements and related results  have been produced recently \cite{brooks2024maximum,gotshall2022spread,li2022maximum,liu2024graph,urschel2021graphs}, the above question remains open. Here, we answer this question in the affirmative, generalizing Theorem \ref{thm:spread} to all non-negative matrices:

\begin{theorem}[Directed Spread Theorem]\label{thm:main}
Let $A$ be an $n \times n$ non-negative matrix. Then
\[ \mathrm{diam}\big(\Lambda(A)\big) \le \frac{2 \,n}{\sqrt{3}}  \, \|A\|_{\max},\]
with equality if and only if $n = 0 \mod 3$ and $A = P_{\pi} \left(\begin{psmallmatrix} 1 & 1 & 1 \\ 1 & 1 & 1 \\ 1 & 1 & 0  \end{psmallmatrix} \otimes J_{\frac{n}{3}}\right)P_{\pi}$, where $J_k$ is the $k \times k$ all-ones matrix and $P_{\pi}$ is a permutation matrix. Furthermore, \begin{enumerate}
    \item the matrix $A$ that maximizes $\mathrm{diam}\big(\Lambda(A)\big) / \|A\|_{\max}$ over all $n \times n$ non-negative matrices is symmetric for all $n$, and
    \item the matrix $A$ that maximizes $\mathrm{diam}\big(\Lambda(A)\big)/\|A\|_{\max}$ over all $n \times n$ zero-diagonal non-negative matrices is exactly the adjacency matrix of the join of a clique on $\lfloor 2n/3 \rfloor$ vertices and an independent set on $\lceil n/3 \rceil$ vertices for all $n$ sufficiently large.
\end{enumerate} \end{theorem}

In contrast to Breen, Riasanovsky, Tait, and Urschel's proof of Theorem \ref{thm:spread} (which involves nearly $70$ pages of complicated mathematics and computation), the proof of Theorem \ref{thm:main} (given Theorem \ref{thm:spread}) is incredibly simple and intuitive. The main idea is that, while the use of the complex plane allows for larger distances in theory, complex eigenvalues come in pairs, making them slightly too ``expensive" for maximizing the spread of eigenvalues.

The key ingredient in our proof is the following lemma regarding the maximum magnitude and real part of non-real eigenvalues of non-negative matrices, which may be of independent interest.

\begin{lemma}\label{lm:bounds}
Let $A$ be an $n \times n$ non-negative matrix with $\|A\|_{\max} \le 1$. Any eigenvalue $\mu$ of $A$ with $\mathrm{Im}(\mu) \neq 0$ satisfies
\begin{equation}\label{ineq:|mu|}
|\mu|^2 \le \min \left\{\frac{n^2+n\lambda_{\max}-2\lambda_{\max}^2- \gamma }{4}, \; \lambda_{\max}^2 \right\}
\end{equation}
and
\begin{equation}\label{ineq:re(mu)}
\mathrm{Re}(\mu)^2 \le \min \left\{\frac{n^2+3n\lambda_{\max}-4\lambda_{\max}^2- \gamma }{8}, \; \lambda_{\max}^2 \right\},
\end{equation}
where $\lambda_{\max}$ is the Perron eigenvalue of $A$ and 
\[\gamma := n\lambda_{\max} - \sum_{i,j = 1}^n \min\{A_{ij},A_{ji}\} + \sum_{i,j=1}^n \left(1+\min\{A_{ij},A_{ji}\}-A_{ij}^2-A_{ji}^2 \right) \ge 0.\]
\end{lemma}

% It will be shown in the proof of Lemma \ref{lm:bounds} that $c$ is always non-negative, so the simpler version of the above bounds with the $-c$ terms removed also holds.

\begin{proof}[Proof of Lemma \ref{lm:bounds}]
First, we note that, by Perron-Frobenius, $\mathrm{Re}(\mu)^2 \le |\mu|^2 \le \lambda_{\max}^2$, immediately handling half of Inequalities \eqref{ineq:|mu|} and \eqref{ineq:re(mu)}.

Next, we verify that $\gamma \ge 0$ and prove Inequality \eqref{ineq:|mu|}. Consider the matrix $\hat{A} \in \mathbb{R}^{n \times n}$ with $\hat{A}_{ij} = \min \{A_{ij}, A_{ji}\}$ for all $i,j \in \{1,\ldots,n\}$. We may express the squared Frobenius norm $\|A\|_F^2 = \sum_{i,j=1}^n A_{ij}^2$ in terms of $n$, $\lambda_{\max}$, and $\gamma$:
\begin{align}
\sum_{i,j=1}^n A_{ij}^2  &= \frac{1}{2} \left(n^2+n\lambda_{\max} - \sum_{i,j=1}^n (1+\min\{A_{ij}, A_{ji}\} - A_{ij}^2 - A_{ji}^2) - (n\lambda_{\max}-\bm{1}^T \Hat{A} \bm{1}) \right) \nonumber \\
&= \frac{1}{2}(n^2 + n\lambda_{\max} - \gamma). \label{ineq:frobenius}
\end{align}
% \textbf{Step 1:} First, we express $\|A\|_F^2$ in terms of $n$, $\lambda_{\max}$, and $\gamma$. Let . Then
% \begin{align}
% \sum_{i,j=1}^n A_{ij}^2 &= \frac{1}{2} \left(n^2+n\lambda_{\max} - \sum_{i,j=1}^n (1+\min\{A_{ij}, A_{ji}\} - A_{ij}^2 - A_{ji}^2) - (n\lambda_{\max}-\bm{1}^T \Hat{A} \bm{1}) \right) \nonumber \\
% &= \frac{1}{2}(n^2 + n\lambda_{\max} - \gamma). \label{ineq:frobenius}
% \end{align}
The quantity $\gamma$ is the sum of the two addends $n\lambda_{\max}-\bm{1}^T \Hat{A} \bm{1}$ and $\sum_{i,j=1}^n (1+\min\{A_{ij}, A_{ji}\} - A_{ij}^2 - A_{ji}^2)$. The former is non-negative by the non-decreasing property of the spectral radius \cite[Corollary 2.1]{Minc88} and the Courant-Fischer theorem:
\begin{equation}\label{ineq:hA_lmax}\lambda_{\max}(A) \ge \lambda_{\max}(\hat{A}) \ge \frac{\textbf{1}^T \hat{A} \textbf{1}}{\textbf{1}^T \textbf{1}} = \frac{\textbf{1}^T \hat{A} \textbf{1}}{n},
\end{equation}
and the latter is non-negative because $\|A\|_{\max} \le 1$. Let $\tilde \Lambda(A)$ denote the multiset of eigenvalues of $A$ with one copy of $\lambda_{\max}$, $\mu$, and $\overline{\mu}$ removed. The lower bound 
\[\|A\|_F^2 \ge \sum_{i=1}^n |\lambda_i|^2 = \lambda_{\max}^2 + 2|\mu|^2 + \sum_{\lambda \in \tilde \Lambda(A)} |\lambda|^2\]
combined with Equation \eqref{ineq:frobenius} produces the remainder of Inequality \eqref{ineq:|mu|}:
\[|\mu|^2 \le \frac{\|A\|_F^2 - \lambda_{\max}^2}{2} \le \frac{n^2+n\lambda_{\max}-2\lambda_{\max}^2-\gamma}{4}.\]
Furthermore, we may also bound the sum of the squared moduli of the remaining eigenvalues:
\begin{align}
\sum_{\lambda \in \tilde \Lambda(A) } |\lambda|^2 &\le \|A\|_F^2 - \lambda_{\max}^2 - 2|\mu|^2 \nonumber \\
&\le \frac{n^2+n\lambda_{\max}-\gamma}{2} - \lambda_{\max}^2 - 2|\mu|^2 \nonumber \\
&= \frac{n^2+n\lambda_{\max}-2\lambda_{\max}^2-\gamma}{2} - 2|\mu|^2.\label{ineq:other-eig}
\end{align}

Finally, using the trace of $A^2$, we prove the remainder of Inequality \eqref{ineq:re(mu)}. We have, by Inequality \eqref{ineq:hA_lmax}, 
\[\mathrm{trace}(A^2) = \sum_{i,j=1}^n A_{ij} A_{ji} \le \sum_{i,j=1}^n \min\{A_{ij}, A_{ji}\} \le n\lambda_{\max}.\]
On the other hand, the trace of $A^2$ can also be written in terms of eigenvalues, giving the lower bound
\begin{align*}
\mathrm{trace}(A^2) &= \lambda_{\max}^2+\mu^2+\overline{\mu}^2+\sum_{\lambda \in \tilde \Lambda(A)} \lambda^2 \\
&\ge \lambda_{\max}^2+\mu^2+\overline{\mu}^2-\sum_{\lambda \in \tilde \Lambda(A)} |\lambda|^2 \\
&\ge \lambda_{\max}^2+\mu^2+\overline{\mu}^2+2|\mu|^2-\frac{n^2+n\lambda_{\max}-2\lambda_{\max}^2-\gamma}{2} \qquad \text{by Inequality \eqref{ineq:other-eig}} \\
&=4\mathrm{Re}(\mu)^2 - \frac{n^2+n\lambda_{\max}-4\lambda_{\max}^2-\gamma}{2}.
\end{align*}
Combining the lower and upper bounds for $\mathrm{trace}(A^2)$ yields our desired result
\begin{align*}
&4\mathrm{Re}(\mu)^2 - \frac{n^2+n\lambda_{\max}-4\lambda_{\max}^2-\gamma}{2} \le n \lambda_{\max} \\
\implies &4\mathrm{Re}(\mu)^2 \le n \lambda_{\max} + \frac{n^2+n\lambda_{\max}-4\lambda_{\max}^2-\gamma}{2} \\
\implies  & \mathrm{Re}(\mu)^2 \le \frac{n^2+3n\lambda_{\max}-4\lambda_{\max}^2-\gamma}{8}.
\end{align*}
\end{proof}

Lemma \ref{lm:bounds} immediately implies that for any eigenvalue $\mu$ with $\mathrm{Im}(\mu) \neq 0$, its modulus and real part are bounded by constant multiples of $n$.

\begin{corollary}\label{co:bounds}
Let $A$ be an $n \times n$ non-negative matrix with $\|A\|_{\max} \le 1$. Any eigenvalue $\mu$ of $A$ with $\mathrm{Im}(\mu) \neq 0$ satisfies
$$|\mu| \le \frac{n}{2} \qquad \text{and} \qquad |\mathrm{Re}(\mu)| \le \frac{3+\sqrt{57}}{24}n. $$
\end{corollary}

\begin{proof}
The desired result follows immediately from Lemma \ref{lm:bounds} and the observations that
\[ \max_{x \in [0,1]} \min \left\{ \frac{1+x-2x^2}{4},x^2 \right\} = \frac{1}{4} \qquad \text{and} \qquad \max_{x \in [0,1]} \min \left\{ \frac{1+3x-4x^2}{8}, x^2 \right\} = \left(\frac{3 + \sqrt{57}}{24}\right)^2,\]
achieved at $x = 1/2$ and $x = (3+\sqrt{57})/24$, respectively.

% By Lemma \ref{lm:bounds}, $|\mu|^2 \le \min \{f(\lambda_{\max}), g(\lambda_{\max}) \}$, where $f(x) = \frac{n^2+nx-2x^2}{4}$ and $g(x) = x^2$. Restricting to $x > 0$, these two functions only intersect at $x = \frac{n}{2}$, which is where $\min \{f(x), g(x)\}$ is maximized. When $x = \frac{n}{2}$, we have $\min \{f(x), g(x)\} = \frac{n^2}{4}$, so $|\mu|^2 \le \frac{n^2}{4}$, as desired.

% By Lemma \ref{lm:bounds}, $\mathrm{Re}(\mu)^2 \le \min\{h(x), g(x)\}$, where $h(x) = \frac{n^2+3nx-4x^2}{8}$ and $g(x)$ is defined as before. Restricting to $x > 0$, these two functions only intersect at $x = \frac{3+\sqrt{57}}{24}n$, which is where $\min \{f(x), g(x)\}$ is maximized. When $x = \frac{3+\sqrt{57}}{24}n$, we have $\min \{f(x), g(x)\} = \left(\frac{3+\sqrt{57}}{24}n\right)^2$, so $\mathrm{Re}(\mu)^2 \le \left(\frac{3+\sqrt{57}}{24}n\right)^2$.
\end{proof}

For the proof of Theorem \ref{thm:main}, we also require the following four lemmas.

\begin{lemma}[Bendixson's inequality, {\cite[Theorem II]{Bendixson}}]\label{lm:symmetric-part-bound}
Let $A \in \mathbb{R}^{n \times n}$. Any eigenvalue $\lambda$ of $A$ satisfies
$ \lambda_{\min}\left(\frac{A+A^T}{2}\right) \le \mathrm{Re}(\lambda) \le \lambda_{\max}\left(\frac{A+A^T}{2}\right)$.
\end{lemma}

\begin{lemma}\label{lm:real-part}
Let $A$ be an $n \times n$ non-negative matrix with $\|A\|_{\max} \le 1$. Then the real part of every eigenvalue of $A$ is at least $ - n /2$.
\end{lemma}
\begin{proof}
This follows immediately from the standard bound $\lambda_{\min}(B) \ge - n/2$ for $n \times n$ symmetric non-negative matrices with entries of magnitude at most one \cite[Proposition 3]{Constantine85} and Lemma \ref{lm:symmetric-part-bound}.
\end{proof}

\begin{lemma}\label{lm:small_spread}
The matrices $\begin{psmallmatrix} 1 & 1 \\ 1 & 0 \end{psmallmatrix}$, $\begin{psmallmatrix} 1 & 1 & 1 \\ 1 & 1 & 1 \\ 1 & 1 & 0 \end{psmallmatrix}$, $\begin{psmallmatrix} 1 & 1 & 1 & 1 \\  1 & 1 & 1 & 1 \\ 1 & 1 & 1 & 1 \\1 & 1 & 1 & 0 \end{psmallmatrix}$, and $\begin{psmallmatrix} 1 & 1 & 1 & 1 & 1 \\  1 & 1 & 1 & 1 & 1\\ 1 & 1 & 1 & 1 & 1\\1 & 1 & 1 & 0 & 0 \\ 1 & 1 & 1& 0 & 0 \end{psmallmatrix}$ have eigenvalue spread $\sqrt{5}$, $2 \sqrt{3}$, $\sqrt{21}$, and $\sqrt{33}$.
\end{lemma}

\begin{lemma}[{\cite[Theorem 7.2.3]{Golub13}}]\label{lm:perturb} Let $A,E \in \mathbb{C}^{n \times n}$, $U^* A U = D + N$ be a Schur decomposition of $A$, where $D$ is diagonal and $N$ is strictly upper triangular, and $\mu \in \Lambda(A+E)$. Then
\begin{equation*}
\min_{\lambda \in \Lambda(A)} |\lambda- \mu| \le \|E\|_2 \sum_{k=0}^{n-1} \|N\|_2^k.
\end{equation*}
\end{lemma}

Using the above lemmas, we are now prepared to prove Theorem \ref{thm:main}.

\begin{proof}[Proof of Theorem \ref{thm:main}]
Suppose, contrary to the theorem statement, that there exists an $n \times n$ non-negative non-symmetric matrix $A$ with $\|A\|_{\max} = 1$ that maximizes $\mathrm{diam}\big(\Lambda(A)\big) / \|A\|_{\max}$ over either all $n \times n$ non-negative matrices or all $n \times n$ zero-diagonal matrices. By Lemma \ref{lm:symmetric-part-bound},
\[\max_{\lambda,\lambda' \in \Lambda(A)} \big|\mathrm{Re}(\lambda) - \mathrm{Re}(\lambda') \big| \le \mathrm{diam}\left( \Lambda \left( \frac{A + A^T}{2} \right) \right).\]
Since $A$ is non-symmetric, $(A+A^T)/2$ is not a $(0,1)$ matrix. Therefore, by Theorem \ref{thm:spread}, the maximum $\mathrm{diam}(\Lambda(A))$ is not achieved by a pair of real-valued eigenvalues. In addition, by Corollary \ref{co:bounds}, the modulus of a non-real eigenvalue is at most $\frac{n}{2}$, so $|\lambda - \lambda'| \le n$ for any two non-real eigenvalues $\lambda$ and $\lambda'$. We note that $n$ is less than both $\frac{2n}{\sqrt{3}} - \frac{1}{2\sqrt{3}n}$ and $\frac{2n-1}{\sqrt{3}}$ for $n \ge 4$, and by Lemma \ref{lm:small_spread} is non-optimal for $n<4$. Therefore, $\mathrm{diam}(\Lambda(A))$ cannot be achieved by a pair of non-real eigenvalues either, and we need only consider the case where one eigenvalue is real and the other is non-real.

Let $\lambda$ be the non-real eigenvalue and $\lambda'$ be the real one. By Lemma \ref{lm:real-part}, $\lambda' \ge -\frac{n}{2}$. If $\lambda' \le \frac{n}{2}$, then both $\lambda$ and $\lambda'$ lie in the disk of radius $\frac{n}{2}$ centered at the origin, so $|\lambda - \lambda'| \le n$, which is not the maximum spread. Otherwise, if $\lambda' > \frac{n}{2}$, then $|\lambda - \lambda'|$ is less than $|\lambda - \lambda_{\max}|$, where $\lambda_{\max}$ is the Perron eigenvalue. Therefore, $\mathrm{diam}(\Lambda(A))$ must be achieved between the Perron eigenvalue, denoted by $\lambda_{\max}$, and some non-real eigenvalue, denoted by $\lambda$.

By Lemma \ref{lm:bounds},
\begin{align*}
| \lambda_{\max} - \lambda|^2 &= \lambda_{\max}^2 - 2 \lambda_{\max} \, \mathrm{Re}(\lambda) + |\lambda|^2 \\
&\le \lambda_{\max}^2 + 2 \lambda_{\max} \sqrt{\frac{n^2+3n\lambda_{\max}- 4 \lambda_{\max}^2 - \gamma}{8}} + \frac{n^2+n\lambda_{\max}- 2 \lambda_{\max}^2 - \gamma}{4} \\
&= n^2 f(\lambda_{\max}/n, \gamma/n^2),
\end{align*}
where 
\[f(x, \eta):= x^2 + 2x\sqrt{\frac{1+3x-4x^2 - \eta}{8}} + \frac{1+x-2x^2-\eta}{4}.\]
By inspection, we have $f(x,0)<1.315609 < 4/3$ for $x \in [0,1]$; the maximum is achieved by setting $\lambda_{\max}$ equal to $n$ times the largest root of $7 + 61 x - 2 x^2 - 592 x^3 + 576 x^4$. Therefore,
\begin{equation}\label{ineq:diam}
\mathrm{diam}\big(\Lambda(A) \big) = | \lambda_{\max} - \lambda| < \sqrt{1.315609 \, n^2} = 1.147 \, n.
\end{equation}
We note that $1.147 \, n$ is strictly less than $2n/\sqrt{3}$ for all $n$, strictly less than $2n/\sqrt{3} - 1/(2 \sqrt{3} n)$ for $n>6$, and strictly less than $(2n-1)/\sqrt{3}$ for all $n$ sufficiently large. Now consider $n \le 6$. By Theorem \ref{thm:spread},  $n =3$ and $n = 6$ achieve the bound $2n/\sqrt{3}$. By Lemma \ref{lm:small_spread}, when $n = 5$ the maximum spread is at least $\sqrt{33} > 1.147\, n$. When $n = 2$, $A$ has no non-real eigenvalues. All that remains is $n =4$.

By Lemma \ref{lm:small_spread}, when $n = 4$ the maximum spread is at least $\sqrt{21}$. By inspection, $f(x,0)<21$ for all $x \not \in [0.85177,0.89726]$ (so $\lambda_{\max}$ must be in $[3.4071,3.58903]$) and $\max_{x \in [0,1]} f(x,1/200)< 21$ (so $\eta$ must be less than $1/200$); see Figure \ref{fig:2plots}. Therefore, $\gamma < 4^2 \times 1/200 = 2/25$. 

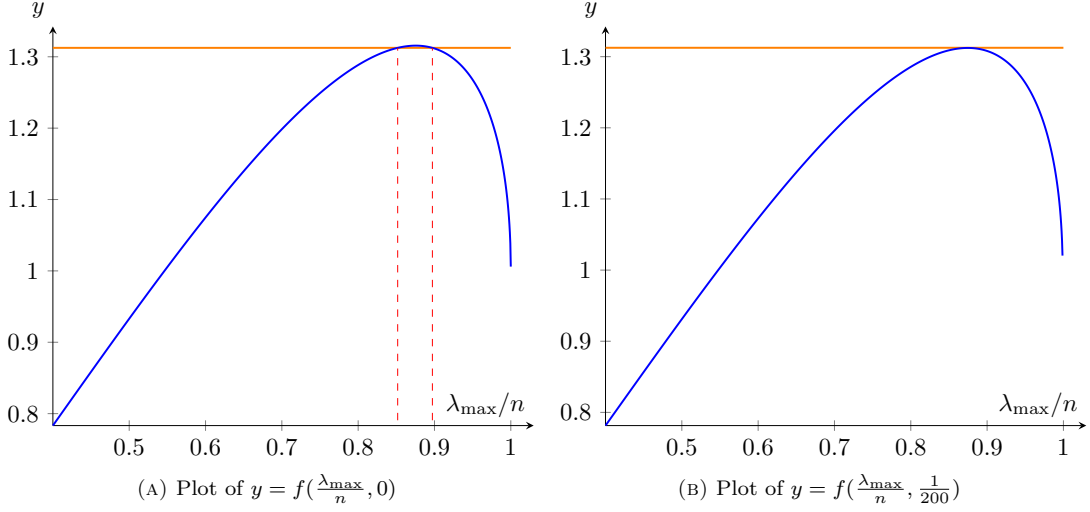
\begin{figure}[t]
\centering
\subfloat[Plot of $y = f(\frac{\lambda_{\max}}{n},0)$]{
  \resizebox{2.8in}{!}{\begin{tikzpicture}
    \begin{axis}[axis lines = center, 
    xlabel = {$\lambda_{\max}/n$}, ylabel = {$y$}, xmax=1.03, ymax=21/16+0.03,
    y label style={at={(axis description cs:-0.03,1)},anchor=south}]
      \addplot[domain=0.4:1,range=0.77:1.4,samples=1000,orange,thick]{21/16};
      \addplot[domain=0.4:1, range=0.77:1.4,samples=1000,blue,thick]{x*x + 2*x*sqrt((1+3*x-4*x*x)/8) + (1+x-2*x*x)/4};
      \draw[red, dashed] (axis cs:0.851774,0.77) -- (axis cs:0.851774,21/16);
      \draw[red, dashed] (axis cs:0.897257,0.77) -- (axis cs:0.897257,21/16);
    \end{axis}
    \end{tikzpicture}}
}
\subfloat[Plot of $y = f(\frac{\lambda_{\max}}{n},\frac{1}{200})$]{
  \resizebox{2.8in}{!}{\begin{tikzpicture}
    \begin{axis}[axis lines = center, 
    xlabel = {$\lambda_{\max}/n$}, ylabel = {$y$}, xmax=1.03, ymax=21/16+0.03,
    y label style={at={(axis description cs:-0.03,1)},anchor=south}]
      \addplot[domain=0.4:1,range=0.77:1.4,samples=1000,orange,thick]{21/16};
      \addplot[domain=0.4:1, range=0.77:1.4,samples=1000,blue,thick]{x*x + 2*x*sqrt((1+3*x-4*x*x-0.005)/8) + (1+x-2*x*x-0.005)/4};
    \end{axis}
    \end{tikzpicture}}
}

\caption{\textsc{(a)} The left plot shows $y = f(\frac{\lambda_{\max}}{n},0)$ (blue) and $y = \frac{21}{16}$ (orange). The blue curve is below the orange line for all $x \not\in [0.85177, 0.89726]$, which is marked by red dashed lines. \textsc{(b)} The right plot shows $y = f(\frac{\lambda_{\max}}{n},\frac{1}{200})$ (blue), which is always below $y = \frac{21}{16}$ (orange). A simple calculation shows that the maximum value of $f(\frac{\lambda_{\max}}{n},\frac{1}{200})$ is approximately 1.31229, which is less than $\frac{21}{16}$.}

\label{fig:2plots}
\end{figure}

By definition, $\gamma \ge 4\lambda_{\max} - \bm{1}^T \hat{A} \bm{1}$, so
$$0.02 > \frac{\gamma}{4} \ge \lambda_{\max} - \frac{\bm{1}^T \hat{A} \bm{1}}{4} \ge 0.$$
Since $\lambda_{\max} \in [3.4071, 3.58903]$, we obtain that $\frac{\bm{1}^T \hat{A} \bm{1}}{4} \in [3.3871, 3.58903]$, which means that the sum of the entries of $\hat{A}$ is in $[13.5484,14.35612]$. Let $\mathrm{round}(\cdot)$ of a matrix denote the matrix resulting from entrywise rounding. In addition,
\begin{align}
    \sum_{i,j = 1}^n \left| \hat A_{ij} - \mathrm{round}(\hat A)_{ij} \right| &\le 2 \sum_{i,j = 1}^n \left|A_{ij} - \mathrm{round}(A)_{ij} \right| \nonumber \\
    &\le 4  \sum_{i,j=1}^n A_{ij}(1-A_{ij}) = 2 \sum_{i,j=1}^n (A_{ij}(1-A_{ij}) + A_{ji}(1-A_{ji})) \nonumber \\
&= 2 \sum_{i,j=1}^n (\max\{A_{ij}, A_{ji}\}+\min\{A_{ij}, A_{ji}\}-A_{ij}^2-A_{ji}^2) \le 2 \gamma < 0.16. \label{ineq:gamma}
\end{align}
Therefore, $\bm{1}^T \mathrm{round}(\hat A) \bm{1} = 14$. The bound $\frac{2}{25} > \gamma \ge 1+A_{ii} -2 A_{ii}^2$ implies that zeros cannot occur on the diagonal, so $\mathrm{round}(\hat{A})$ has the following form, up to row and column permutations:
$$\mathrm{round}(\hat{A}) = \begin{pmatrix}
1 & 1 & 1 & 1 \\
1 & 1 & 1 & 1 \\
1 & 1 & 1 & 0 \\
1 & 1 & 0 & 1
\end{pmatrix}. $$
The bound $\frac{2}{25} > \gamma \ge 1+\min\{A_{ij}, A_{ji}\} - A_{ij}^2 - A_{ji}^2 \ge 1 - \max\{A_{ij}, A_{ji}\}^2$ implies that $\max\{A_{ij}, A_{ji}\} \in [0.95917, 1]$, so the rounded value of the maximum entry equals 1. As a result, there is only a single possibility for $\mathrm{round}(A)$, up to reflection across its diagonal:
\begin{equation*}
\mathrm{round}(A) = \begin{pmatrix}
1 & 1 & 1 & 1 \\
1 & 1 & 1 & 1 \\
1 & 1 & 1 & 1 \\
1 & 1 & 0 & 1
\end{pmatrix}.    
\end{equation*}
The spectrum of this matrix $\mathrm{round}(A)$ is $(0, 0, 2-\sqrt{3}, 2+\sqrt{3})$. To finish the proof, we apply the eigenvalue perturbation bound from Lemma \ref{lm:perturb} to $\mathrm{round}(A)$. Let $E = A - \mathrm{round}(A)$. Given that the sum of squared magnitudes of the eigenvalues of $\mathrm{round}(A)$ is $14$ and the squared Frobenius norm of $\mathrm{round}(A)$ is $15$, the Frobenius norm of the strictly upper-triangular part of the Schur decomposition of $\mathrm{round}(A)$ is $\|N\|_F = 1$. In addition, by Inequality \eqref{ineq:gamma},
\[ \|E\|_2 \le \|E\|_F \le \sum_{i,j = 1}^n \left|A_{ij} - \mathrm{round}(A)_{ij} \right| \le \gamma.\]
By Lemma \ref{lm:perturb}, for any $\mu \in \Lambda(A)$,
$$ \min_{\lambda \in \Lambda(\mathrm{round}(A))} |\lambda - \mu| \le \|E\|_2 \sum_{k=0}^{n-1} \|N\|_2^k \le \|E\|_F \sum_{k=0}^{n-1} \|N\|_F^k \le 4 \gamma. $$
Therefore,
\[\mathrm{diam}(\Lambda(A)) \le \mathrm{diam}(\Lambda(\mathrm{round}(A))) + 8 \gamma < (2+\sqrt{3}) + 16/25 < \sqrt{21},\]
a contradiction.
\end{proof}

\section*{Acknowledgements}
This material is based upon work supported by the National Science Foundation under grant no. DMS-2513687.
This manuscript is the result of a summer research project through the MIT Research Science Institute (RSI). The first author would like to thank Dr. Tanya Khovanova, Genaro Laymuns, AnaMaria Perez, Prof. Roman Bezrukavnikov, and Dr. Jonathan Bloom for providing helpful feedback on early versions of this project. The first author would also like to thank the Center for Excellence in Education (CEE) and MIT for organizing and funding this research opportunity. The authors are grateful to Louisa Thomas for improving the style of presentation. 

{ \small 
	\bibliographystyle{plain}
	\bibliography{main.bib} }

\end{document}